\documentclass[a4paper, 11pt,]{article}

\usepackage{amsmath,amsthm}
\usepackage{amsfonts,amssymb}
\usepackage{graphicx}
\usepackage{calc}
\usepackage{indentfirst}
\usepackage{multicol}
\def\normo#1{\left\|#1\right\|}
\def\abs#1{|#1|}
\def\aabs#1{\big|#1\big|}
\def\brk#1{\left(#1\right)}

\def\norm#1{\|#1\|}
\def\jb#1{\langle#1\rangle}

\newcommand{\N}{{\mathbb N}}
\newcommand{\T}{{\mathbb T}}

\newcommand{\R}{{\mathbb R}}

\newcommand{\Z}{{\mathbb Z}}
\newcommand{\ft}{{\mathcal{F}}}

\newcommand{\les}{{\lesssim}}

\newcommand{\Sch}{{\mathcal{S}}}

\numberwithin{equation}{section}

\newtheorem{theorem}{Theorem}[section]

\newtheorem{proposition}[theorem]{Proposition}

\newtheorem{lemma}[theorem]{Lemma}
\newtheorem{corollary}[theorem]{Corollary}

\newtheorem{remark}[theorem]{Remark}

\begin{document}
\title{\bf Global well-posedness and inviscid limit for the modified Korteweg-de
Vries-Burgers equation}

\author{\bf Hua Zhang  \date{}\\
{\small \it LMAM, School of Mathematical Sciences, Peking
University, Beijing 100871, China}\\ {\small E-mail:
zhanghuamaths@163.com} }\maketitle

\maketitle


{\bf Abstract:} Considering the Cauchy problem for the modified
Korteweg-de Vries-Burgers equation
\begin{eqnarray*}
u_t+u_{xxx}+\epsilon |\partial_x|^{2\alpha}u=2(u^{3})_x, \ \
u(0)=\phi,
\end{eqnarray*}
where $0<\epsilon,\alpha\leq 1$ and $u$ is a real-valued function,
we show that it is uniformly globally well-posed in $H^s\ (s\geq1)$
for all $\epsilon \in (0,1]$. Moreover, we prove that for any $s\geq
1$ and $T>0$, its solution converges in $C([0,T]; \,H^s)$ to that of
the MKdV equation if $\epsilon$ tends to $0$.



\section{Introduction}
In this paper, we study the Cauchy problem for the modified
Korteweg-de Vries-Burgers (MKdV-B) equation with fractional
dissipation
\begin{eqnarray}
u_t+u_{xxx}+\epsilon |\partial_x|^{2\alpha}u=2(u^{3})_x, \ \
u(0)=\phi,\label{eq:MKdV-B}
\end{eqnarray}
where $0<\epsilon, \alpha \leq 1$, $u$ is a real-valued function of
$(x, t) \in \mathbb{R}\times \mathbb{R}_+$. The equation with
quadratic nonlinearity
\begin{eqnarray}
u_t+u_{xxx}+\epsilon |\partial_x|^{2\alpha}u=2(u^{2})_x, \ \
u(0)=\phi,\label{eq:kdvb}
\end{eqnarray}
has been derived as a model for the propagation of weakly nonlinear
dispersive long waves in some physical contexts when dissipative
effects occur (see \cite{OS}). On the other hand, the cubic
nonlinearity is also of much interest.

The Cauchy problems \eqref{eq:MKdV-B} and \eqref{eq:kdvb} has been
studied by many authors (see \cite{WL,WL2,MR,MR3,Guo,Guo2} and the
reference therein). In \cite{MR} Molinet and Ribaud studied Eq.
\eqref{eq:kdvb} in the case $\alpha=1$ and showed that
\eqref{eq:kdvb} is globally well-posed in $H^{s}\ (s>-1)$ by using
an $X^{s,b}$-type space which contains the dissipative structure.
Their result is sharp in the sense that the solution map of
\eqref{eq:kdvb} fails to be $C^2$ smooth at origin if $s<-1$. Their
result is generalized to the case $0<\alpha\leq1$ by Vento
\cite{Vento}, also by Guo and Wang \cite{Guo} and found a critical
wellposedness regularity
\begin{eqnarray}\label{eq:sa}
s_\alpha=\left \{
\begin{array}{ll}
-3/4, &  0<\alpha\leq1/2,\\
-3/(5-2\alpha),&  1/2<\alpha\leq 1.
\end{array}
\right.
\end{eqnarray}
In \cite{Guo}, Guo and Wang also proved a uniform global
wellposedness in $H^s\ (s>-3/4)$ and that the solution converges in
$C([0,T];H^s)$ to that of the KdV equation for any $T>0$ when
$\epsilon$ tends to zero, by using a $l^1-variant$ $X^{s,b}$ space
and I-method. For the Eq. \eqref{eq:MKdV-B}, following the methods
in \cite{MR}, Chen and Li \cite{WL} showed global wellposedness in
$H^{s}, s>-1/4$ and Chen, Li and Miao \cite{WL2} obtained in $H^{s}(
s>1/4-\alpha/4)$ for the case $0<\alpha\leq 1$.

Following the ideas in \cite{Guo}, we consider the inviscid limit of
Eq. \eqref{eq:MKdV-B} as $\epsilon$ tends to zero. Formally, if
$\epsilon=0$ then \eqref{eq:MKdV-B} reduces to the MKdV equation
\begin{eqnarray}
u_t+u_{xxx}=6u^{2}u_x, \ \ u(0)=\phi.\label{eq:MKdV}
\end{eqnarray}
The optimal result on local well-posedness of \eqref{eq:MKdV} in
$H^s$ was obtained by Kenig, Ponce and Vega \cite{KPV}. They
obtained that \eqref{eq:MKdV} is locally well-posed for $s\geq1/4$.
 The result on global well-posedness of \eqref{eq:MKdV} in $H^s$ was obtained  in
\cite{Tao2} where it was shown that \eqref{eq:MKdV} is globally
well-posed in $H^s$ for $s>1/4$ and a kind of modified energy
method, so called I-method, is introduced. It is natural to
conjecture that the solution of Eq. \eqref{eq:MKdV-B} will converge
to that of Eq. \eqref{eq:MKdV} if $\epsilon$ tends to zero. To prove
that, we prove first the uniform global well-posedness of Eq.
\eqref{eq:MKdV-B}. Then we need to study the difference equation
between \eqref{eq:MKdV-B} and \eqref{eq:MKdV}. We first treat the
dissipative term as perturbation and then use the uniform Lipschitz
continuity property of the solution map. Similar ideas can be found
in \cite{Wang} for the inviscid limit of the complex Ginzburg-Landau
equation. For $T>0$, we denote $S_T^{\epsilon}$, $S_T$ the solution
map of \eqref{eq:MKdV-B}, \eqref{eq:MKdV} respectively. The notation
 $A\lesssim B$ denotes that there exists a constant $C$, such that $A\leq CB$.  Now we state our main results.

\begin{theorem}\label{t12}
Assume $0<\alpha\leq 1$ and $s\geq1$. Let $\phi \in H^s(\R)$. Then
for any $T>0$, the solution map $S_T^\epsilon$ satisfies for all
$0<\epsilon \leq 1$
\begin{equation}
\norm{S_T^\epsilon \phi}_{F^s(T)}\les C(T,\norm{u}_{H^s})
\end{equation}
where $F^s(T)\subset C([0,T];H^s)$ which will be defined later and
$C(\cdot,\cdot)$ is a continuous function with $C(\cdot,0)=0$, and
also satisfies that for all $0<\epsilon \leq 1$
\begin{eqnarray}
\norm{S_T^\epsilon (\phi_1)-S_T^\epsilon
(\phi_2)}_{C([0,T],H^s)}\leq
C(T,\norm{\phi_1}_{H^1},\norm{\phi_2}_{H^1})\norm{\phi_1-\phi_2}_{H^1}.
\end{eqnarray}
\end{theorem}

We also have the uniform persistence of regularity, following the
standard argument. For local well-posedness we actually prove that
complex-valued Eq. \eqref{eq:MKdV-B} is uniformly locally well-posed
in $H^s(s\geq 1/4)$. For the limit behavior, we have
\begin{theorem}\label{t13}
Assume $0<\alpha\leq 1$. Let $\phi \in H^s(\R)$, $s\geq1$. For any
$T>0$, then
\begin{equation}\label{eq:limitthm}
\lim_{\epsilon\rightarrow
0^+}\norm{S_T^\epsilon(\phi)-S_T(\phi)}_{C([0,T], H^s)}=0.
\end{equation}
\end{theorem}

\begin{remark} \label{hr}
We are only concerned with the limit in the same regularity space.
There seems no convergence rate. This can be seen from the linear
solution,
\begin{equation}
\norm{e^{-t\partial_{x}^3-t\epsilon|\partial_x|^{2\alpha}}\phi-e^{-t\partial_{x}^3}\phi}_{C([0,T],H^s)}\rightarrow
0, \quad \mbox{as }\epsilon\rightarrow 0,
\end{equation}
but without any convergence rate. We believe that there is a
convergence rate if we assume the initial data has higher regularity
than the limit space. For example, we prove that
\begin{equation}
\norm{S_T^\epsilon(\phi_1)-S_T(\phi_2)}_{C([0,T],H^1)}\les
\norm{\phi_1-\phi_2}_{H^1}+\epsilon^{1/2}C(T,\norm{\phi_1}_{H^2},\norm{\phi_2}_{H^1}).
\end{equation}
\end{remark}

The rest of the paper is organized as following. We present some
notations and Banach function spaces  in Section 2. We give a
symmetric estimate in Section 3. We prove the trilinear estimate in
Section 4. We present uniform LWP in Section 5 and prove Theorem
\ref{t12} in Section 6. Theorem \ref{t13}  is proved in Section 7.

\section{Notation and Definitions} \label{notation}
For $x, y\in \R$, $x\sim y$ means that there exist $C_1, C_2 > 0$
such that $C_1|x|\leq |y| \leq C_2|x|$. For $f\in \Sch'$ we denote
by $\widehat{f}$ or $\ft (f)$ the Fourier transform of $f$ for both
spatial and time variables,
\begin{eqnarray*}
\widehat{f}(\xi, \tau)=\int_{\R^2}e^{-ix \xi}e^{-it \tau}f(x,t)dxdt.
\end{eqnarray*}
We denote  by $\ft_x$ the the Fourier transform on spatial variable
and if there is no confusion, we still write $\ft=\ft_x$. Let
$\mathbb{Z}$ and $\mathbb{N}$ be the sets of integers and natural
numbers, respectively.  For the simplicity, we let $\Z_{+}= \N
\bigcup \{0\}$. For $k\in \Z$, let
 $I_k=\{\xi: |\xi|\in [2^{k-1}, 2^{k+1}]\ \}$. For $k\in \Z_+$ let
 $\widetilde{I}_k=[-2,2]$ if $k=0$ and $\widetilde{I}_k=I_k$ if
$k\geq 1$.  For $k\in \Z_+$ and $j\geq 0$ let $D_{k,j}=\{(\xi,
\tau)\in \R \times \R: \xi \in \widetilde{I}_k, \tau-\omega(\xi)\in
\widetilde{I}_j\}$. For $k\in \Z$ and $j\geq 0$ let
$\dot{D}_{k,j}=\{(\xi, \tau)\in \R \times \R: \xi \in {I}_k,
\tau-\omega(\xi)\in \widetilde{I}_j\}$.    We use $f*g$ will stand
for the convolution on time and spatial variables, i.e.,
$$
 (f* g)
(t,x)= \int_{\R^2} f(t-s,x-y) g(s,y)dsdy.
$$

 Let $\eta_0: \R\rightarrow
[0, 1]$ denote an even smooth function supported in $[-8/5, 8/5]$
and equal to $1$ in $[-5/4, 5/4]$. For $k\in \N$ let
$\eta_k(\xi)=\eta_0(\xi/2^k)-\eta_0(\xi/2^{k-1})$ and $\eta_{\leq
k}=\sum_{k'=0}^k\eta_{k'}$. For $k\in \Z$ let
$\chi_k(\xi)=\eta_0(\xi/2^k)-\eta_0(\xi/2^{k-1})$. Roughly speaking,
$\{\chi_k\}_{k\in \mathbb{Z}}$ is the homogeneous decomposition
function sequence and $\{\eta_k\}_{k\in \mathbb{Z}_+}$ is the
non-homogeneous decomposition function sequence to the frequency
space. For $k\in \Z_+$ let $P_k$ denote the operator on $L^2(\R)$
defined by \[ \widehat{P_ku}(\xi)=\chi_k(\xi)\widehat{u}(\xi)\]   By
a slight abuse of notation we also define the operator $P_k$ on
$L^2(\R\times \R)$ by the formula $\ft(P_ku)(\xi,
\tau)=\chi_k(\xi)\ft (u)(\xi, \tau)$. For $l\in \Z$ let
\[
P_{\leq l}=\sum_{k\leq l}P_k, \quad P_{\geq l}=\sum_{k\geq l}P_k.
\]

We define the Lebesgue spaces $L_T^qL_x^p$ and $L_x^pL_T^q$ by the
norms
\begin{equation}
\norm{f}_{L_T^qL_x^p}=\normo{\norm{f}_{L_x^p}}_{L_t^q([0,T])}, \quad
\norm{f}_{L_x^pL_T^q}=\normo{\norm{f}_{L_t^q([0,T])}}_{L_x^p}.
\end{equation}
We denote by $W_0$ the semigroup associated with Airy-equation
\[
\ft_x(W_0(t)\phi)(\xi)=\exp[i\xi^3t]\widehat{\phi}(\xi), \  \forall
\ t\in \R,\ \phi \in \mathcal {S}'.
\]
For $0<\epsilon\leq 1$ and $0< \alpha \leq 1$, we denote by
$W_\epsilon^\alpha$ the semigroup associated with the free evolution
of \eqref{eq:MKdV-B},
\[
\ft_x(W_\epsilon^\alpha(t)\phi)(\xi)=\exp[-\epsilon
|\xi|^{2\alpha}t+i\xi^3t]\widehat{\phi}(\xi),\ \forall \ t\geq 0,\
\phi \in \mathcal {S}',
\]
and we extend $W_\epsilon^\alpha$ to a linear operator defined on
the whole real axis by setting
\[
\ft_x(W_\epsilon^\alpha(t)\phi)(\xi)=\exp[-\epsilon
|\xi|^{2\alpha}|t|+i\xi^3t]\widehat{\phi}(\xi),\ \forall \ t\in \R,
\ \phi \in \mathcal {S}'.
\]
Let
\begin{eqnarray} \label{defL}
L(f)(x,t)=2W_0(t)\int_{\R^2}e^{ix\xi}\frac{e^{it\tau'}-e^{-\epsilon
|t||\xi|^{2\alpha}}}{i\tau'+\epsilon
|\xi|^{2\alpha}}\ft(W_0(-t)f)(\xi,\tau')d\xi d\tau'.
\end{eqnarray}
To study the low regularity of \eqref{eq:kdvb}, Molinet and Ribaud
introduced the variant version of Bourgain's spaces with dissipation
\begin{equation}\label{eq:Fs}
\norm{u}_{X^{b,s,\alpha}}=\norm{\jb{i(\tau-\xi^3)+|\xi|^{2\alpha}}^b\jb{\xi}^s\widehat{u}}_{L^2(\R^2)},
\end{equation}
where $\jb{\cdot}=(1+|\cdot|^2)^{1/2}$. The standard $X^{b,s}$ space
for \eqref{eq:MKdV} used by Bourgain, Kenig, Ponce and  Vega (see
\cite{Bour}, \cite{KPV} ) is defined by
\begin{eqnarray*}
\norm{u}_{X^{b,s}}=\norm{\jb{\tau-\xi^3}^b\jb{\xi}^s\widehat{u}}_{L^2(\R^2)}.
\end{eqnarray*}
We introduce the Banach spaces used in \cite{Guo}. For $k\in \Z_+$
we define the dyadic $X^{b,s}$-type normed spaces $X_k=X_k(\R^2)$,
\begin{eqnarray*}
X_k=\{f\in L^2(\R^2): &&f(\xi,\tau) \mbox{ is supported in }
I_k\times\R \mbox{ and }\\&& \norm{f}_{X_k}=\sum_{j=0}^\infty
2^{j/2}\norm{\eta_j(\tau-\xi^3)\cdot f}_{L^2}\}.
\end{eqnarray*}
This kind of spaces were introduced, for instance, in \cite{IKT},
\cite{Tata} and  \cite{In-Ke} for the BO equation. From the
definition of $X_k$, we see that for any $l\in \Z_+$ and $f_k\in
X_k$ (see also \cite{IKT}),
\begin{equation}
\sum_{j=0}^\infty 2^{j/2}\norm{\eta_j(\tau-\xi^3)\int
|f_k(\xi,\tau')|2^{-l}(1+2^{-l}|\tau-\tau'|)^{-4}d\tau'}_{L^2}\les
\norm{f_k}_{X_k}.
\end{equation}
Hence for any $l\in \Z_+$, $t_0\in \R$, $f_k\in X_k$  and $\gamma
\in \Sch(\R)$, then
\begin{equation}
\norm{\ft[\gamma(2^l(t-t_0))\cdot \ft^{-1}f_k]}_{X_k}\les
\norm{f_k}_{X_k}.
\end{equation}
For $s\geq 0$, we define the following spaces:
\begin{eqnarray}
&&F^{s}= \{u\in \mathcal {S}'(\R^2):\ \norm{u}_{F^s}^2=\sum_{k \in
\Z_+}2^{2sk}\norm{\eta_k(\xi)\ft(u)}_{X_k}^2<\infty\},\\
&& N^{s}=  \{u\in \mathcal {S}'(\R^2): \norm{u}_{N^s}^2=\sum_{k\in
\Z_+}2^{2sk}\norm{(i+\tau-\xi^3)^{-1}\eta_k(\xi)\ft(u)}_{X_k}^2<\infty\}.
\end{eqnarray}
For $T\geq 0$, we define the time-localized spaces
$X_T^{b,s,\alpha}$, $X_T^{b,s}$, $F^{s}(T)$ and $N^{s}(T)$
\begin{eqnarray}
&& \norm{u}_{X_T^{b,s,\alpha}}=\inf_{w\in
X^{b,s,\alpha}}\{\norm{w}_{X^{b,s,\alpha}},\  w(t)=u(t) \mbox{ on }
[0, T]\};\nonumber \\
&&\norm{u}_{X_T^{b,s}}=\inf_{w\in X^{b,s}}\{\norm{w}_{X^{b,s}}, \ \
w(t)=u(t) \mbox{ on } [0, T]\};\nonumber\\
&&\norm{u}_{F^{s}(T)}=\inf_{w\in F^{s}}\{\norm{w}_{F^{s}}, \
w(t)=u(t)
\mbox{ on } [0, T]\};\nonumber\\
&&\norm{u}_{N^{s}(T)}=\inf_{w\in N^{s}}\{\norm{w}_{N^{s}}, \
w(t)=u(t) \mbox{ on } [0, T]\}.
\end{eqnarray}

As a conclusion to this section, we recall a result in \cite{Guo}.
\begin{proposition}[Proposition 2.1, \cite{Guo}]\label{p21}
Let $Y$ be a Banach space of functions on $\R\times \R$ with the
property that
\[\norm{e^{it\tau_0}e^{-t\partial_x^3}f}_Y\les \ \norm{f}_{H^s(\R)}\]
holds for all $f\in H^s(\R)$ and $\tau_0\in \R$. Then we have the
embedding
\begin{equation}
\left(\sum_{k\in \Z_+}\norm{P_k u}_{Y}^2 \right)^{1/2}\lesssim
\norm{u}_{F^s}.
\end{equation}
\end{proposition}

\section {A symmetric estimate}

According to the standard fixed point argument, we will need the
following trilinear estimate.
\begin{lemma}\label{l41}
If $s\geq \frac{1}{4}$, then exists $C>0$, such that for any
$u,v,w\in F^s$
\begin{eqnarray}
\norm{\partial_x(uvw)}_{N^{s}}\leq
C(\norm{u}_{F^{s}}\norm{v}_{F^{\frac{1}{4}}}\norm{w}_{F^{\frac{1}{4}}}+\norm{v}_{F^{s}}\norm{u}_{F^{\frac{1}{4}}}\norm{w}_{F^{\frac{1}{4}}}+\norm{w}_{F^{s}}\norm{v}_{F^{\frac{1}{4}}}\norm{u}_{F^{\frac{1}{4}}})
\end{eqnarray}
\end{lemma}
Now we prove a symmetric estimate which will be used to prove the
trilinear estimate, closely following the methods in \cite{Guo2}.
Similar ideas for the bilinear estimates can be found in
\cite{In-Ke}. For $\xi_1, \xi_2, \xi_3\in \R$ and $\omega:\R
\rightarrow \R$ defined as $\omega(\xi)=\xi^{3}$. Let
\begin{equation}\label{eq:reso}
\Omega(\xi_1,\xi_2,\xi_3)=\omega(\xi_1)+\omega(\xi_2)+\omega(\xi_3)-\omega(\xi_1+\xi_2+\xi_3).
\end{equation}
This is the resonance function that plays a crucial role in the
trilinear estimate of the $X^{s,b}$-type space, see \cite{Taokz} for
a perspective discussion. For compactly supported functions $f,g,h,u
\in L^2(\R\times \R)$. Let
\begin{eqnarray*}
&&J(f,g,h,u)=\int_{\R^6}f(\xi_1,\mu_1)g(\xi_2,\mu_2)h(\xi_3,\mu_3)\\
&&u(\xi_1+\xi_2+\xi_3,\mu_1+\mu_2+\mu_3+\Omega(\xi_1,\xi_2,\xi_3))d\xi_1d\xi_2d\xi_3d\mu_1d\mu_2d\mu_3.
\end{eqnarray*}

\begin{lemma}\label{l42}
Assume $k_1,k_2,k_3,k_4 \in \Z$,  $k_1\leq k_2\leq k_3\leq k_4$,
$j_1,j_2,j_3,j_4\in \Z_+$ and $f_{k_i,j_i}\in L^2(\R\times \R)$ are
nonnegative functions supported in $I_{k_i}\times
\widetilde{I}_{j_i}, \ i=1,\ 2,\ 3,\ 4$. For simplicity we write
$J=|J(f_{k_1,j_1},f_{k_2,j_2},f_{k_3,j_3},f_{k_4,j_4})|$.

(a) For any $k_1\leq k_2\leq k_3\leq k_4$ and $j_1,j_2,j_3,j_4\in
\Z_+$,
\begin{equation}
J\leq C 2^{(j_{min}+j_{thd})/2}2^{(k_{min}+k_{thd})/2}
\prod_{i=1}^4\norm{f_{k_i,j_i}}_{L^2}.
\end{equation}

(b) If $k_2\leq k_3-5$ and $j_2\neq j_{max}$,
\begin{equation}\label{eq:l41rb}
J\leq C
2^{(j_1+j_2+j_3+j_4)/2}2^{-j_{max}/2}2^{-k_{max}}2^{k_{min}/2}
\prod_{i=1}^4\norm{f_{k_i,j_i}}_{L^2};
\end{equation}
if $k_2\leq k_3-5$ and $j_2=j_{max}$,
\begin{equation}
J\leq C
2^{(j_1+j_2+j_3+j_4)/2}2^{-j_{max}/2}2^{-k_{max}}2^{k_{thd}/2}
\prod_{i=1}^4\norm{f_{k_i,j_i}}_{L^2}.
\end{equation}

(c) For any $k_1,k_2,k_3,k_4 \in \N$  and $j_1,j_2,j_3,j_4\in \Z_+$,
\begin{equation}
 J\leq C2^{(j_1+j_2+j_3+j_4)/2}2^{-j_{max}/2}2^{-(k_{1}+k_{2}+k_{3})/6}
 \prod_{i=1}^4\norm{f_{k_i,j_i}}_{L^2}.
 \end{equation}

(d) If $k_{min}\leq k_{max}-10$, then
\begin{equation}
J\leq C2^{(j_1+j_2+j_3+j_4)/2}2^{-3k_{max}/2}
\prod_{i=1}^4\norm{f_{k_i,j_i}}_{L^2}.
\end{equation}
Here we use $k_{max},k_{sec},k_{thd}$ and $k_{min}$ denote the
maximum, the second maximum , the third maximum number and the
 minimum of numbers $k_1,k_2,k_3$ and $k_4$. The notations
$j_{max},j_{sec},j_{thd}$ and $j_{min}$ are similar.
\end{lemma}
\begin{proof}
Let $A_{k_i}(\xi)=[\int_\mathbb{R}
\abs{f_{k_i,j_i}(\xi,\mu)}^{2}d\mu]^{\frac{1}{2}}, i=1,2,3,4$, then
$\norm{A_{k_i}}_{L^{2}_{\xi}}=\norm{f_{k_{i},j_{i}}}_{L^{2}_{\xi,\mu}}$.
Using the Cauchy-Schwartz inequality and the support properties of
the functions $f_{k_i,j_i}$,
\begin{eqnarray*}
&&|J(f_{k_1,j_1},f_{k_2,j_2},f_{k_3,j_3},f_{k_4,j_4})|\\
&\leq& C
2^{(j_{min}+j_{thd})/2}\int_{\R^3}A_{k_1}(\xi_1)A_{k_2}(\xi_1)A_{k_3}(\xi_1)A_{k_4}(\xi_1+\xi_2+\xi_3)d\xi_1d\xi_2d\xi_3\\
&\leq& C
2^{(k_{min}+k_{thd})/2}2^{(j_{min}+j_{thd})/2}\prod_{i=1}^4\norm{A_{k_i,j_i}}_{L^2},
\end{eqnarray*}
which is part (a), as desired.

For part (b), by examining the supports of the functions,
$J(f_{k_1,j_1},f_{k_2,j_2},f_{k_3,j_3},f_{k_4,j_4})\equiv 0$ unless
\begin{equation}\label{eq:freeq}
k_4\leq k_3+5.
\end{equation}
Simple changes of variables in the integration and the observation
that the function $\omega$ is odd show that
\begin{eqnarray}\label{ss}
|J(f,g,h,u)|=|J(g,f,h,u)|=|J(f,h,g,u)|=|J(\widetilde{f},\widetilde{g},u,h)|,
\end{eqnarray}
where $\widetilde{f}(\xi,\mu)=f(-\xi,-\mu),
\widetilde{g}(\xi,\mu)=g(-\xi,-\mu)$. We assume first that $j_2\neq
j_{max}$ and $j_4=j_{max}$, then we will prove that if
$g_i:\R\rightarrow \R_+$ are $L^2$ nonnegative functions supported
in $I_{k_i}$, $i=1,2,3$ and $g: \R^2\rightarrow \R_+$ is an $L^2$
function supported in $I_{k_4}\times \widetilde{I}_{j_{4}}$, then
\begin{eqnarray}\label{eq:l41b}
&&\int_{\R^3}g_1(\xi_1)g_2(\xi_2)g_3(\xi_3)g(\xi_1+\xi_2+\xi_3,\Omega(\xi_1,\xi_2,\xi_3))d\xi_1d\xi_2d\xi_3\nonumber\\
&&\les
2^{(j_1+j_2+j_3)}2^{-k_{max}}2^{k_{min}/2}\norm{g_1}_{L^2}\norm{g_2}_{L^2}\norm{g_3}_{L^2}\norm{g}_{L^2}.
\end{eqnarray}
This suffices for \eqref{eq:l41rb}.

To prove \eqref{eq:l41b}, we first observe that since $k_2\leq
k_3-5$ then $|\xi_3+\xi_2|\sim |\xi_3|$. By change of variables
$\xi'_1=\xi_1$,  $\xi'_2=\xi_2$,  $\xi'_3=\xi_2+\xi_3$, we get that
the left side of \eqref{eq:l41b} is dominated by
\begin{eqnarray}\label{eq:l41b2}
&&\int_{|\xi'_1|\sim 2^{k_1},|\xi'_2|\sim 2^{k_2},|\xi'_3|\sim
2^{k_3}}g_1(\xi'_1)g_2(\xi'_2)\nonumber\\
&&g_3(\xi'_3-\xi'_2)g(\xi'_1+\xi'_3,\Omega(\xi'_1,\xi'_2,\xi'_3-\xi'_2))d\xi'_1d\xi'_2d\xi'_3.
\end{eqnarray}
Note that in the integration area we have
\begin{eqnarray*}
\big|\frac{\partial}{\partial_{\xi'_2}}\left[\Omega(\xi'_1,\xi'_2,\xi'_3-\xi'_2)\right]\big|=|\omega'(\xi'_2)-\omega'(\xi'_3-\xi'_2)|\sim
2^{2k_3},
\end{eqnarray*}
where we use the fact $\omega'(\xi)\sim|\xi|^{2}$ and $k_2\leq
k_3-5$. So we have
$\norm{g(\xi'_{1}+\xi'_{3},\Omega(\xi'_{1},\xi'_{2},\xi'_{3}-\xi'_{2}))}_{L^{2}_{\xi'_{2}}}=2^{-k_{3}}\norm{g(\xi'_{1}+\xi'_{3},\mu_{2})}_{L^{2}_{\mu_{2}}}$.
By change of variable $\mu_2=\Omega(\xi'_1,\xi'_2,\xi'_3-\xi'_2)$,
we get that \eqref{eq:l41b2} is dominated by
\begin{eqnarray}
&&2^{-k_3}\int_{|\xi'_1|\sim
2^{k_1}}g_1(\xi'_1)\norm{g_2}_{L^2}\norm{g_3}_{L^2}\norm{g}_{L^2}d\xi'_1 \nonumber\\
&\les&2^{-k_{max}}2^{k_{min}/2}\norm{g_1}_{L^2}\norm{g_2}_{L^2}\norm{g_3}_{L^2}\norm{g}_{L^2}.
\end{eqnarray}

If $j_3=j_{max}$,  this case is identical to the case $j_4=j_{max}$
in view of \eqref{ss}. If $j_1=j_{max}$, similar to \eqref{eq:l41b},
it suffices to prove $g_i:\R\rightarrow \R_+$ are $L^2$  nonnegative
functions supported in $I_{k_i}$, $i=2,3,4$  and $g: \R^2\rightarrow
\R_+$ is an $L^2$ nonnegative function supported in $I_{k_1}\times
\widetilde{I}_{j_1}$, then
\begin{eqnarray}\label{eq:l41b3}
&&\int_{\R^3}g_2(\xi_2)g_3(\xi_3)g_4(\xi_4)g(\xi_2+\xi_3+\xi_4,\Omega(\xi_2,\xi_3,\xi_4))d\xi_2d\xi_3d\xi_4\nonumber\\
&&\les
2^{-k_{max}}2^{k_{min}/2}\norm{g_2}_{L^2}\norm{g_3}_{L^2}\norm{g_4}_{L^2}\norm{g}_{L^2}.
\end{eqnarray}

Indeed, by change of variables $\xi'_2=\xi_2,  \xi'_3=\xi_3,
 \xi'_4=\xi_2+\xi_3+\xi_4$ and the observation that in the area
$|\xi'_2|\sim 2^{k_2},|\xi'_3|\sim 2^{k_3},|\xi'_4|\sim 2^{k_1}$,
\begin{eqnarray*}
\big|\frac{\partial}{\partial_{\xi'_2}}\left[\Omega(\xi'_2,\xi'_3,\xi'_4-\xi'_2-\xi'_3)\right]\big|=|\omega'(\xi'_2)-\omega'(\xi'_4-\xi'_2-\xi'_3)|\sim
2^{2k_3},
\end{eqnarray*}
we get from Cauchy-Schwartz inequality that
\begin{eqnarray}
&&\int_{\R^3}g_2(\xi_2)g_3(\xi_3)g_4(\xi_4)g(\xi_2+\xi_3+\xi_4,\Omega(\xi_2,\xi_3,\xi_4))d\xi_2d\xi_3d\xi_4\nonumber\\
&\les& \int_{|\xi'_2|\sim 2^{k_2},|\xi'_3|\sim 2^{k_3}, |\xi'_4|\sim
2^{k_1}}g_2(\xi'_2)g_3(\xi'_3)\nonumber\\
&&\quad \cdot
g_4(\xi'_4-\xi'_2-\xi'_3)g(\xi'_4,\Omega(\xi'_2,\xi'_3,\xi'_4-\xi'_2-\xi'_3))d\xi'_2d\xi'_3d\xi'_4\nonumber\\
&\les&2^{-k_3}\int_{|\xi'_3|\sim 2^{k_3}, |\xi'_4|\sim
2^{k_1}}g_3(\xi'_3)\norm{g_2(\xi'_2)g_4(\xi'_4-\xi'_2-\xi'_3)}_{L_{\xi'_2}^2}\norm{g(\xi'_4,\cdot)}_{L_{\xi'_2}^2}d\xi'_3d\xi'_4\nonumber\\
&\les&2^{-k_{max}}2^{k_{min}/2}\norm{g_2}_{L^2}\norm{g_3}_{L^2}\norm{g_4}_{L^2}\norm{g}_{L^2}.
\end{eqnarray}

We assume now that $j_2=j_{max}$. This case is identical to the
case $j_1=j_{max}$ in view of \eqref{ss}. We note that we actually
prove that if
 $k_2\leq k_3-5$ then
\begin{eqnarray}\label{eq:l41bb}
J\leq C
2^{(j_1+j_3+j_4)/2}2^{-k_{max}}2^{k_{thd}/2}\prod_{i=1}^4\norm{f_{k_i,j_i}}_{L^2}.
\end{eqnarray}
Therefore, we complete the proof of part (b).

For part (c), setting
$f^{\sharp}_{k_i,j_i}(\xi,\tau)=f_{k_i,j_i}(\xi,
\tau-\omega(\xi))$, $i=1,2,3,4 $,  then we get
\begin{eqnarray*}
&&|J(f_{k_1,j_1},f_{k_2,j_2},f_{k_3,j_3},f_{k_4,j_4})|\\
&=&\left|\int_{\R^{6}}\prod_{i=1}^{3}f^\sharp_{k_i,j_i}(\xi_{i},\tau_{i})f^\sharp_{k_4,j_4}(\xi_{1}+\xi_{2}+\xi_{3},\tau_{1}+\tau_{2}+\tau_{3})d\xi_{1}d\xi_{2}d\xi_{3}d\tau_{1}d\tau_{2}d\tau_{3}\right|
\end{eqnarray*}
Making variables change
$\xi=\xi_{1}+\xi_{2}+\xi_{3},\tau=\tau_{1}+\tau_{2}+\tau_{3}$, we
have
\begin{eqnarray*}
&&|J(f_{k_1,j_1},f_{k_2,j_2},f_{k_3,j_3},f_{k_4,j_4})|\\
&=&\left|\int_{\R^2}f^\sharp_{k_1,j_1}*f^\sharp_{k_2,j_2}*f^\sharp_{k_3,j_3}(\xi,\tau)\cdot f^\sharp_{k_4,j_4}(\xi,\tau)d\xi d\mu\right|\\
&\les& \norm{f^\sharp_{k_1,j_1}\ast f^\sharp_{k_2,j_2}\ast
f^\sharp_{k_3,j_3}}_{L^2} \norm{f^\sharp_{k_4,j_4}}_{L^2}\\
&\les&\prod_{i=1}^{3}\norm{\ft^{-1}f^\sharp_{k_i,j_i}}_{L^{6}_{t,x}}\norm{f^\sharp_{k_4,j_4}}_{L^{2}_{t,x}}
\end{eqnarray*}
On the other hand
\begin{eqnarray*}
\ft^{-1}(f^\sharp_{k_i,j_i})&=&\int_{\R^2}f_{k_i,j_i}(\xi,
\tau-\omega(\xi))e^{ix\xi}e^{it\tau}d\xi d\tau\\
&=&\int_{\R^2} f_{k_i,j_i}(\xi,
\tau)e^{ix\xi}e^{it\omega(\xi)}e^{it\tau}d\xi d\tau,
\end{eqnarray*}
If we can show that
\begin{eqnarray}
\norm{\ft^{-1}f^\sharp_{k_i,j_i}}_{L^{6}_{t,x}}
&\les&\int_{\R}\norm{\int_{\R}f_{k_{i},j_{i}}(\xi,\tau)e^{ix\xi}e^{it\omega(\xi)}d\xi}_{L^{6}_{t,x}}d\tau\nonumber\\
&\les&2^{j_{i}/2}2^{-k_{i}/6}\norm{f_{k_{i},j_{i}}}_{L^2}\label{sstr}
\end{eqnarray}
then part (c) follows by  symmetry.

  Now we give the proof of (\ref{sstr}) by using a result of Wang \cite{Wangbook}.
\begin{lemma}\label{mstr}
Let $U_{m}=e^{it(-\Delta)^{m/2}}, m\geq2$, define
\begin{align}
m^{*}=\left\{
\begin{array}{ll}
\infty, & n\leq m,\\
2n/(n-m)
 , & n>m,
\end{array}
\right.
 \label{defm}
\end{align}
and
\begin{equation}
\frac{2}{\gamma(\cdot)}=n(\frac{1}{2}-\frac{1}{\cdot}),
\end{equation}
\begin{equation}
2\sigma(m,\cdot)=n(2-m)(\frac{1}{2}-\frac{1}{\cdot}).
\end{equation}

For $2\leq r, p <2^{*}$, we have
\begin{eqnarray}
\norm{U_{m}(t)\phi}_{L^{\gamma(p)}(I,\dot{B}^{s-\sigma(m,p)}_{p,2})}&\lesssim&
\norm{\phi}_{\dot{H}^{s}}\\
\norm{A_{U_{m}}(t)f}_{L^{\gamma(p)}(I,\dot{B}^{s-\sigma(m,p)}_{p,2})}&\lesssim&
\norm{f}_{L^{\gamma(r)'}(I,\dot{B}^{s+\sigma(m,r)}_{r',2})}
\end{eqnarray}
Here $I \subset
 \R$ is a any interval, $A_{U_{m}}:=\int_{0}^{t}U_{m}(t-\tau)\cdot
 d\tau$.
\end{lemma}
Choosing $m=3, n=1 ,s=0, p=6$ and $r=2$ in Lemma \ref{mstr}, we get
\eqref{sstr}.

 For part (d), we need to consider two cases: $\xi_1\cdot \xi_2>0$
or $\xi_1\cdot \xi_2<0$. Observing that if we let
$\xi_{4}=-(\xi_{1}+\xi_{2}+\xi_{3})$, then
$\Omega(\xi_{1},\xi_{2},\xi_{3})=3(\xi_{1}+\xi_{2})(\xi_{1}+\xi_{3})(\xi_{1}+\xi_{4})$.
The former case is easier to handle. When $\xi_1\cdot \xi_2>0$,
because of $k_{1}\leq k_{4}-10$,  we have
$\Omega(\xi_{1},\xi_{2},\xi_{3})\geq2^{k_{2}}2^{k_{3}}2^{k_{4}}\sim2^{k_{2}+2k_{3}}$.

If  $k_{2}\leq k_{3}-5$, notice that $j_{max}\geq k_{2}+2k_{3}-20$,
part (d) holds by part (b).

 If $k_{2}\geq k_{3}-5$, observing $-\frac{k_{1}+k_{2}+k_{3}}{6}\sim -\frac{k_{1}+2k_{3}}{6}$ and owning to $j_{max}\geq
k_{2}+2k_{3}-20$, part (d) holds by part (c).

 We assume now
$\xi_1\cdot \xi_2<0$. Now we should consider serval cases according
to $j_{i}=j_{max}$. If $j_4=j_{max}$, it suffices to prove that if
$A_i$ is $L^2$ nonnegative functions supported in $I_{k_i}$,
$i=1,2,3$  and $B$ is a $L^2$ nonnegative function supported in
$I_{k_4}\times \widetilde{I}_{j_4}$,  then
\begin{eqnarray}\label{eq:l41d1}
&&\int_{\R^3\cap \{\xi_1\cdot
\xi_2<0\}}A_1(\xi_1)A_2(\xi_2)A_3(\xi_3)B(\xi_1+\xi_2+\xi_3,
\Omega(\xi_1,\xi_2,\xi_3))d\xi_1d\xi_2d\xi_3\nonumber\\
&\les&
2^{j_4/2}2^{-2k_3}\norm{A_1}_{L^2}\norm{A_2}_{L^2}\norm{A_3}_{L^2}\norm{B}_{L^2}.
\end{eqnarray}
By localizing $|\xi_1+\xi_2|\sim 2^l$ for $l\in \Z$, we get that the
right-hand side of \eqref{eq:l41d1} is dominated by
\begin{eqnarray}\label{eq:l41d2}
\sum_{l}\int_{\R^3}\chi_{l}(\xi_1+\xi_2)A_1(\xi_1)A_2(\xi_2)A_3(\xi_3)B(\xi_1+\xi_2+\xi_3,
\Omega(\xi_1,\xi_2,\xi_3))d\xi_1d\xi_2d\xi_3.
\end{eqnarray}
From the support properties of the functions $A_i,\ B$ and the fact
that in the integration area
\[|\Omega(\xi_1,\xi_2,\xi_3)|=|3(\xi_1+\xi_2)(\xi_1+\xi_3)(\xi_2+\xi_3)|\sim
2^{l+2k_3},\] We get that
\begin{equation}\label{eq:l41dsi}
j_{max}\geq l+2k_3-20.
\end{equation}
By change of variables $\xi_1'=\xi_1+\xi_2$, $\xi_2'=\xi_2$,
$\xi_3'=\xi_1+\xi_3$,  we obtain that \eqref{eq:l41d2} is dominated
by
\begin{eqnarray}\label{eq:l41d3}
&&\sum_{l}\int_{|\xi_1'|\sim 2^l,|\xi_2'|\sim 2^{k_2},|\xi_3'|\sim 2^{k_3}}\chi_{l}(\xi_1')A_1(\xi_1'-\xi_2')A_2(\xi_2')A_3(\xi_2'+\xi_3'-\xi_1')\nonumber\\
&&\quad B(\xi_2'+\xi_3',
\Omega(\xi_1'-\xi_2',\xi_2',\xi_2'+\xi_3'-\xi_1'))d\xi_1'd\xi_2'd\xi_3'.
\end{eqnarray}
Since in the integration area
\begin{eqnarray}\label{eq:l41djo}
\big|\frac{\partial}{\partial_{\xi_1'}}[\Omega(\xi_1'-\xi_2',\xi_2',\xi_2'+\xi_3'-\xi_1')]\big|
=|\omega'(\xi_1'-\xi_2')-\omega'(\xi_2'+\xi_3'-\xi_1')|\sim2^{2k_3},
\end{eqnarray}
then we get from \eqref{eq:l41djo} that \eqref{eq:l41d3}  is
dominated by
\begin{eqnarray}
&&\sum_{l}\int_{|\xi_1'|\sim 2^l}\chi_{l}(\xi_1')\norm{A_1}_{L^2}\norm{A_3}_{L^2}\nonumber\\
&&\quad \norm{A_2(\xi_2')B(\xi_2'+\xi_3',
\Omega(\xi_1'-\xi_2',\xi_2',\xi_2'+\xi_3'-\xi_1'))}_{L^2_{\xi_2',\xi_3'}}d\xi_1'\nonumber\\
&\les&\sum_{l}2^{l/2}2^{-k_3}\norm{A_1}_{L^2}\norm{A_2}_{L^2}\norm{A_3}_{L^2}\norm{B}_{L^2}\nonumber\\
&\les&
2^{j_{max}/2}2^{-2k_3}\norm{A_1}_{L^2}\norm{A_2}_{L^2}\norm{A_3}_{L^2}\norm{B}_{L^2},
\end{eqnarray}
where we used \eqref{eq:l41dsi} in the last inequality. From
symmetry we know the case $j_3=j_{max}$ is identical to the case
$j_4=j_{max}$;  the case $j_1=j_{max}$ is identical to the case
$j_2=j_{max}$. Thus it reduces to prove the case $j_2=j_{max}$. It
suffices to prove that if $A_i$ is $L^2$ nonnegative functions
supported in $I_{k_i}$, $i=1,3,4$ and $B$ is a $L^2$ nonnegative
function supported in $I_{k_2}\times \widetilde{I}_{j_2}$, then
\begin{eqnarray}\label{eq:l41dc21}
&&\int_{\R^3\cap \{\xi_1\cdot
\xi_2<0\}}A_1(\xi_1)A_3(\xi_3)A_4(\xi_4)B(\xi_1+\xi_3+\xi_4,
\Omega(\xi_1,\xi_3,\xi_4))d\xi_1d\xi_3d\xi_4\nonumber\\
&\les&
2^{j_2/2}2^{-2k_{3}}\norm{A_1}_{L^2}\norm{A_4}_{L^2}\norm{A_3}_{L^2}\norm{B}_{L^2}.
\end{eqnarray}
As the case $j_4=j_{max}$, we get that the right-hand side of
\eqref{eq:l41dc21} is dominated by
\begin{eqnarray}\label{eq:l41dc22}
\sum_{l}\int_{\R^3}\chi_{l}(\xi_3+\xi_4)A_1(\xi_1)A_4(\xi_4)A_3(\xi_3)B(\xi_1+\xi_4+\xi_3,
\Omega(\xi_1,\xi_3,\xi_4))d\xi_1d\xi_4d\xi_3.
\end{eqnarray}
From the support properties of the functions $A_i,\ B$ and the
fact that in the integration area
\[|\Omega(\xi_1,\xi_2,\xi_3)|=|3(\xi_1+\xi_4)(\xi_2+\xi_4)(\xi_3+\xi_4)|\sim 2^{l+2k_3},\]
We get that
\begin{equation}\label{eq:l41dc2si}
j_{max}\geq l+2k_3-20.
\end{equation}
By changing variables $\xi_1'=\xi_1+\xi_3$, $\xi_3'=\xi_3+\xi_4$,
$\xi_4'=\xi_1+\xi_3+\xi_4$, we obtain that \eqref{eq:l41dc22} is
dominated by
\begin{eqnarray}\label{eq:l41dc23}
&&\sum_{l}\int_{|\xi_3'|\sim 2^l,|\xi_4'|\sim 2^{k_2},|\xi_1'|\sim 2^{k_3}}\chi_{l}(\xi_3')A_1(\xi_4'-\xi_3')A_3(\xi_1'+\xi_3'-\xi_4')A_4(\xi_4'-\xi_1')\nonumber\\
&&\quad
B(\xi_4',\Omega(\xi_4'-\xi_3',\xi_1'+\xi_3'-\xi_4',\xi_4'-\xi_1'))d\xi_1'd\xi_3'd\xi_4'.
\end{eqnarray}
Since in the integration area,
\begin{eqnarray}\label{eq:l41dc2jo}
&&\big|\frac{\partial}{\partial_{\xi_3'}}[\Omega(\xi_4'-\xi_3',\xi_1'+\xi_3'-\xi_4',\xi_4'-\xi_1')]\big|\nonumber\\
&=&|-\omega'(\xi_4'-\xi_3')+\omega'(\xi_1'+\xi_3'-\xi_4')|\sim
2^{2k_3},
\end{eqnarray}
then we get from \eqref{eq:l41dc2jo} that \eqref{eq:l41dc23} is
dominated by
\begin{eqnarray}
&&\sum_{l}\int_{|\xi_3'|\sim 2^l}\chi_{l}(\xi_3')\norm{A_1}_{L^2}\norm{A_3}_{L^2}\nonumber\\
&&\quad \norm{A_4(\xi_4'-\xi_1')B(\xi_4',
\Omega(\xi_4'-\xi_3',\xi_1'+\xi_3'-\xi_4',\xi_4'-\xi_1'))}_{L^2_{\xi_1',\xi_4'}}d\xi_3'\nonumber\\
&\les&\sum_{l}2^{l/2}2^{-k_3}\norm{A_1}_{L^2}\norm{A_3}_{L^2}\norm{A_4}_{L^2}\norm{B}_{L^2}\nonumber\\
&\les&
2^{j_{max}/2}2^{-2k_3}\norm{A_1}_{L^2}\norm{A_3}_{L^2}\norm{A_4}_{L^2}\norm{B}_{L^2},
\end{eqnarray}
where we used \eqref{eq:l41dc2si} in the last inequality. Therefore,
we complete the proof of part (d).
\end{proof}
We restate  Lemma \ref{l42} in a form that is suitable for the
trilinear estimates in the next sections.
\begin{corollary}\label{cor42}
Assume $k_1,k_2,k_3,k_4\in \Z$, $j_1,j_2,j_3,j_4\in \Z_+$ and
$f_{k_i,j_i}\in L^2(\R\times \R)$ are functions supported in
$\dot{D}_{k_i,j_i}$, $i=1,2,3$.

(a) For any $k_1,k_2,k_3,k_4\in \Z$ and $j_1,j_2,j_3,j_4\in \Z_+$,
\begin{eqnarray}
&&\norm{1_{\dot{D}_{k_4,j_4}}(\xi,\tau)(f_{k_1,j_1}*f_{k_2,j_2}*f_{k_3,j_3})}_{L^2}\nonumber\\
&\leq&
C2^{(k_{min}+k_{thd})/2}2^{(j_{min}+j_{thd})/2}\prod_{i=1}^3\norm{f_{k_i,j_i}}_{L^2}.
\end{eqnarray}
(b) For any $k_1,k_2,k_3,k_4\in \Z$ with $k_{thd}\leq k_{sec}-5$ and
$j_1,j_2,j_3,j_4\in \Z_+$. If for some $i\in \{1,2,3,4\}$ such that
$(k_i,j_i)=(k_{thd},j_{max})$, then
\begin{eqnarray}
&&\norm{1_{\dot{D}_{k_4,j_4}}(\xi,\tau)(f_{k_1,j_1}*f_{k_2,j_2}*f_{k_3,j_3})}_{L^2}\nonumber\\
&\leq&
C2^{-k_{max}}2^{-k_{thd/2}}2^{(j_{1}+j_{2}+j_3+j_4)/2}2^{-j_{max}/2}\prod_{i=1}^3\norm{f_{k_i,j_i}}_{L^2},
\end{eqnarray}

else we have
\begin{eqnarray}
&&\norm{1_{\dot{D}_{k_4,j_4}}(\xi,\tau)(f_{k_1,j_1}*f_{k_2,j_2}*f_{k_3,j_3})}_{L^2}\nonumber\\
&\leq&
C2^{-k_{max}}2^{k_{min}/2}2^{(j_{1}+j_{2}+j_3+j_4)/2}2^{-j_{max}/2}\prod_{i=1}^3\norm{f_{k_i,j_i}}_{L^2}.
\end{eqnarray}
(c) For any $k_1,k_2,k_3,k_4\in \N$ and $j_1,j_2,j_3,j_4\in \Z_+$,
\begin{eqnarray}
&&\norm{1_{\dot{D}_{k_4,j_4}}(\xi,\tau)(f_{k_1,j_1}*f_{k_2,j_2}*f_{k_3,j_3})}_{L^2}\nonumber\\
&\leq&
C2^{(j_{1}+j_{2}+j_3+j_4)/2}2^{-j_{max}/2}2^{-(k_{min}+k_{thd}+k_{sec})/6}\prod_{i=1}^3\norm{f_{k_i,j_i}}_{L^2}.
\end{eqnarray}
(d) For any $k_1,k_2,k_3,k_4\in \Z$ with $k_{min}\leq k_{max}-10$
and $j_1,j_2,j_3,j_4\in \Z_+$,
\begin{eqnarray}
&&\norm{1_{\dot{D}_{k_4,j_4}}(\xi,\tau)(f_{k_1,j_1}*f_{k_2,j_2}*f_{k_3,j_3})}_{L^2}\nonumber\\
&\leq&
C2^{(j_{1}+j_{2}+j_3+j_4)/2}2^{-3k_{max}/2}\prod_{i=1}^3\norm{f_{k_i,j_i}}_{L^2}.
\end{eqnarray}
\end{corollary}
\begin{proof}
Clearly, we have
\begin{eqnarray}
&&\norm{1_{\dot{D}_{k_4,j_4}}(\xi,\tau)(f_{k_1,j_1}*f_{k_2,j_2}*f_{k_3,j_3})(\xi,\tau)}_{L^2}\nonumber\\
&=&\sup_{\norm{f}_{L^2}=1}\aabs{\int_{\dot{D}_{k_4,j_4}} f\cdot
f_{k_1,j_1}*f_{k_2,j_2}*f_{k_3,j_3} d\xi d\tau}.
\end{eqnarray}
Let $f_{k_4,j_4}=1_{\dot{D}_{k_4,j_4}}\cdot f$, we have
$f_{k_i,j_i}^\sharp(\xi,\mu)=f_{k_i,j_i}(\xi,\mu+\omega(\xi))$,
$i=1,2,3,4$. The functions $f_{k_i,j_i}^\sharp$ are supported in
$I_{k_i}\times \bigcup_{|m|\leq 3}\widetilde{I}_{j_i+m}$,
$\norm{f_{k_i,j_i}^\sharp}_{L^2}=\norm{f_{k_i,j_i}}_{L^2}$. Using
simple changes of variables,  we get
\[\int_{\dot{D}_{k_4,j_4}} f\cdot
f_{k_1,j_1}*f_{k_2,j_2}*f_{k_3,j_3} d\xi d\tau =
J(f_{k_1,j_1}^\sharp,f_{k_2,j_2}^\sharp,f_{k_3,j_3}^\sharp,f_{k_4,j_4}^\sharp).\]
Then Corollary \ref{cor42} follows from Lemma \ref{l42}.
\end{proof}
\section{Proof of the trilinear estimate}
In this section, we use the following propositions to prove Lemma
\ref{l41}. For simplicity, we let
$$G=2^{k_{4}}\norm{\eta_{k_{4}}(\xi)(\tau-\omega(\xi)+i)^{-1} f_{k_{1}}*f_{k_{2}}* f_{k_{3}}}_{X_{k_{4}}}$$
\begin{proposition}\label{p51}
For $0\leq k_1\leq k_2\leq k_3-10, k_3\geq 110, |k_4-k_3|\leq 5$, we
have \\
$G\lesssim2^{k_{1}/2}\prod\limits_{i=1}^3\norm{f_{k_i}}_{X_{k_i}}$.
\begin{proof}
According to the definition of $X_{k}$
\begin{eqnarray*}
G&\leq& 2^{k_{4}}\sum_{j_{4}=0}^\infty2^{j_{4}/2}\norm{1_{\dot{D}_{k_4,j_4}}(\xi,\tau)(\tau-\omega(\xi)+i)^{-1}(f_{k_1,j_1}*f_{k_2,j_2}*f_{k_3,j_3})}_{L^2}\\
&\les&2^{k_{4}}\sum_{j_{4}=0}^\infty2^{j_{4}/2}2^{-j_{4}}\norm{1_{\dot{D}_{k_4,j_4}}(\xi,\tau)(f_{k_1,j_1}*f_{k_2,j_2}*f_{k_3,j_3})}_{L^2}\\
&\les&2^{k_{4}}\sum_{j_{4}=0}^\infty2^{-j_{4}/2}2^{-3k_{max}/2}2^{(j_{1}+j_{2}+j_{3}+j_{4})/2}\prod_{i=1}^3\norm{f_(k_{i},j_{i})}_{L^{2}}\\
&\les&2^{k_{1}/2}\prod_{i=1}^3\norm{f_{k_i}}_{X_{k_i}}
\end{eqnarray*}
Here we use the Corollary \ref{cor42} (d) and
$(\tau-\omega(\xi)+i)^{-1}\sim2^{-j_4}$.
\end{proof}
\end{proposition}
\begin{proposition}\label{p52}
For $0\leq k_1\leq k_2\leq k_3, k_{2}\geq k_{3}-10, k_3\geq 110,
|k_4-k_3|\leq5,k_1\leq k_2-10$,\\ we  have
$G\lesssim2^{k_{1}/2}\prod\limits_{i=1}^3\norm{f_{k_i}}_{X_{k_i}}$.
\begin{proof}
The proof is similar to Proposition \ref{p51},  we omit it.
\end{proof}
\end{proposition}
\begin{proposition}\label{p53}
For $0\leq k_1\leq k_2\leq k_3, k_1\geq k_3-30, k_3\geq 110,
|k_4-k_3|\leq 5$,\\ we  have
$G\lesssim2^{k_{1}/2}\prod\limits_{i=1}^{3}
\norm{f_{k_i}}_{X_{k_i}}$.
\begin{proof}
According to Corollary \ref{cor42} (c), we have
\begin{eqnarray*}
&&\norm{1_{\dot{D}_{k_4,j_4}}(\xi,\tau)(f_{k_1,j_1}*f_{k_2,j_2}*f_{k_3,j_3})}_{L^2}\\
&&\lesssim
2^{(j_{1}+j_{2}+j_3+j_4)/2}2^{-j_{max}/2}2^{-(k_{min}+k_{thd}+k_{sec})/6}\prod_{i=1}^3\norm{f_{k_i,j_i}}_{L^2}.
\end{eqnarray*}
So we get
\begin{eqnarray*}
G&\lesssim&
2^{k_{4}}\sum_{j_{4}=0}^{\infty}2^{-j_{4}/2}2^{(j_{1}+j_{2}+j_3+j_4)/2}2^{-j_{max}/2}2^{-(k_{min}+k_{thd}+k_{sec})/6}\prod_{i=1}^3\norm{f_{k_i,j_i}}_{L^{2}}\\
&\lesssim&
2^{(k_{1}+k_{2})/4}\prod_{i=1}^3\norm{f_{k_i}}_{X_{k_i}}
\end{eqnarray*}
In the last inequality, we use the fact that $k_{min}\geq
k_{max}-40$ and $\abs{k_{4}-k_{3}}\leq5$.
\end{proof}
\end{proposition}
\begin{proposition}\label{p54}
For  $ 0\leq k_1\leq k_2\leq k_3$, $k_4\leq k_3-10, k_3\geq 110$,
 $|k_2-k_3|\leq 5 $,\\ when $k_{1}\leq k_{2}-6$, we have
$G\lesssim2^{k_{1}/2}\prod\limits_{i=1}^3\norm{f_{k_i}}_{X_{k_i}}$
;\\ when $|k_1-k_2|\leq5$, we have
$G\lesssim2^{(k_{1}+k_{2})/4}\prod\limits_{i=1}^3\norm{f_{k_i}}_{X_{k_i}}$.
\begin{proof}
Firstly, we consider  the case $k_{1}\leq k_{2}-10$. Using the
method in prove lemma \ref{l42} and  making variables change
$\xi'_{1}=\xi_{1}, \xi'_{2}=\xi_{1}+\xi_{2}, \xi'_{3}=\xi_{3} $, we
can easily get
\begin{eqnarray*}
&&\norm{1_{\dot{D}_{k_4,j_4}}(\xi,\tau)(f_{k_1,j_1}*f_{k_2,j_2}*f_{k_3,j_3})}_{L^2}\\
&&\lesssim
2^{-k_{max}}2^{k_{1}/2}2^{(j_{1}+j_{2}+j_{3}+j_{4})/2}2^{-j_{max}/2}\prod_{i=1}^3\norm{f_{k_i}}_{L^{2}}
\end{eqnarray*}
By this result, we  get
\begin{eqnarray*}
G&\lesssim&2^{k_{4}}\sum_{j_{4}=0}^{\infty}2^{-j_{4}/2}2^{-k_{max}}2^{k_{1}/2}2^{(j_{1}+j_{2}+j_{3}+j_{4})/2}
2^{-j_{max}/2}\prod_{i=1}^3\norm{f_{k_i}}_{L^{2}}\\
&\lesssim&2^{k_{1}/2}\prod_{i=1}^3\norm{f_{k_i}}_{X_{k_i}}
\end{eqnarray*}
Secondly, we consider the case ${|k_{1}-k_{2}|\leq5}$. Observing
in this case, we have $k_{min}\geq k_{max}-30$, the result we need
follows by Corollary \ref{cor42} (c).
\end{proof}
\end{proposition}
\begin{proposition}\label{p55}
For $0\leq k_1\leq k_2\leq k_3, \max(k_3,k_4)\leq 120$,   we have \\
$G\lesssim2^{(k_{1}+k_{2})/4}\prod\limits_{i=1}^3\norm{f_{k_i}}_{X_{k_i}}$
\begin{proof}
Using Corollary \ref{cor42} (a), noticing that $k_{i}\leq 120,
i=1,2,3,4$, we have
\begin{eqnarray*}
G&\lesssim&
2^{(k_{min}+k_{thd})/2}\prod_{i=1}^3\norm{f_{k_i}}_{X_{k_i}}\\
&\lesssim& 2^{(k_{1}+k_{2})/4}\prod_{i=1}^3\norm{f_{k_i}}_{X_{k_i}}
\end{eqnarray*}
\end{proof}
\end{proposition}
Now we turn to the proof of Lemma \ref{l41}.\\
\begin{bfseries}{Proof}:
\end{bfseries}
In view of definition, we get
$$\norm{\partial_x(uvw)}_{N^{s}}^{2}=\sum_{k_{4}=0}^{\infty}2^{2sk_{4}}\norm{\eta_{k_{4}}(\xi)(\tau-\omega(\xi)+i)^{-1}
\ft[\partial_{x}(uvw)]}_{X_{k_{4}}}^{2}$$  For the simplicity of
notation, setting $f_{k_1}=\eta_{k_1}(\xi)\ft(u)(\xi,\tau)$,
$f_{k_2}=\eta_{k_2}(\xi)\ft(v)(\xi,\tau)$,  and
$f_{k_3}=\eta_{k_3}(\xi)\ft(w)(\xi,\tau)$,  for $k_1, k_2, k_3\in
\Z_+$, then we get
\begin{eqnarray*}
&&\norm{\eta_{k_4}(\xi)(\tau-\omega(\xi)+i)^{-1}\ft
[\partial_{x}(uvw)]}_{X_{k_4}}\\
&&\les \sum_{k_1,k_2,k_3\in\Z_+}\norm{\xi\cdot\eta_{k_4}(\xi)(\tau-\omega(\xi)+i)^{-1}f_{k_1}*f_{k_2}*f_{k_3}}_{X_{k_4}}\\
&&\les \sum_{k_1,k_2,k_3\in
\Z_+}2^{k_4}\norm{\eta_{k_4}(\xi)(\tau-\omega(\xi)+i)^{-1}f_{k_1}*f_{k_2}*f_{k_3}}_{X_{k_4}}.
\end{eqnarray*}
From symmetry, it suffices to bound
\begin{eqnarray*}
\sum_{0\leq k_1\leq k_2\leq
k_3}2^{k_4}\norm{\eta_{k_4}(\xi)(\tau-\omega(\xi)+i)^{-1}f_{k_1}*f_{k_2}*f_{k_3}}_{X_{k_4}}.
\end{eqnarray*}
Dividing the summation into the several parts,  we get
\begin{eqnarray}\label{eq:trilinear}
&&\sum_{k_1\leq k_2\leq
k_3}2^{k_4}\norm{\eta_{k_4}(\xi)(\tau-\omega(\xi)+i)^{-1}f_{k_1}*f_{k_2}*f_{k_3}}_{X_{k_4}}\nonumber\\
&\leq& \sum_{j=1}^5\sum_{(k_1,k_2,k_3,k_4)\in A_j}
2^{k_4}\norm{\eta_{k_4}(\xi)(\tau-\omega(\xi)+i)^{-1}f_{k_1}*f_{k_2}*f_{k_3}}_{X_{k_4}},
\end{eqnarray}
where we denote
\begin{eqnarray*}
&&A_1=\{0\leq k_1\leq k_2\leq k_3-10, k_3\geq 110, |k_4-k_3|\leq 5 \};\\
&&A_2=\{0\leq k_1\leq k_2\leq k_3-10, k_3\geq 110, |k_4-k_3|\leq
5,
k_1\leq k_2-10 \};\\
&&A_3=\{0\leq k_1\leq k_2\leq k_3, k_1\geq k_3-30, k_3\geq 110,
|k_4-k_3|\leq 5 \};\\
&&A_4=\{0\leq k_1\leq k_2\leq k_3, k_4\leq k_3-10, k_3\geq 110,
|k_2-k_3|\leq 5\};\\
&&A_5=\{0\leq k_1\leq k_2\leq k_3, \max(k_3,k_4)\leq 120\}.
\end{eqnarray*}
We will apply Proposition \ref{p51}-\ref{p55} obtained in the
beginning of this section to bound the five terms in
\eqref{eq:trilinear}. For example,  for the first term, from
Proposition \ref{p51},  we have
\begin{eqnarray*}
&&\norm{2^{sk_4}\sum_{k_i\in
A_1}2^{k_4}\norm{\eta_{k_4}(\xi)(\tau-\omega(\xi)+i)^{-1}f_{k_1}*f_{k_2}*f_{k_3}}_{X_{k_4}}}_{l_{k_4}^2}\\
&\lesssim& \norm{2^{sk_4}\sum_{k_i\in
A_1}2^{(k_1)/2}\norm{f_{k_1}}_{X_{k_1}}\norm{f_{k_2}}_{X_{k_2}}\norm{f_{k_3}}_{X_{k_3}}}_{l_{k_4}^2}\\
&\lesssim& \norm{u}_{F^{1/4}}\norm{v}_{F^{1/4}}\norm{w}_{F^s}.
\end{eqnarray*}
For the other terms we can handle them in the same way. Therefore we
complete the proof of the Lemma \ref{l41}.

\section{Uniform LWP for MKdV-B equation}
In this section we study the uniform local well-posedness for the
MKdV-Burgers equation. We will prove a time localized version of
Theorem \ref{t12} where $T=T(\norm{\phi}_{H^{s}})$ is small. In
\cite{WL2}, the result of three estimate in  $X^{b,s}$ space is
depend on $\alpha, \varepsilon$ , so it is not proper in this
situation. To get a uniform result about $\alpha, \varepsilon$, we
will use the space $F^s$. Let us recall that \eqref{eq:MKdV-B} is
invariant in the following scaling
\begin{equation}\label{eq:scaling}
u(x,t)\rightarrow \lambda u(\lambda x, \lambda^3 t),\
\phi(x)\rightarrow \lambda\phi(\lambda x), \ \epsilon\rightarrow
\lambda^{4-2\alpha}\epsilon, \ \ \forall \ 0<\lambda\leq 1.
\end{equation}
This invariance is very important in the proof of Theorem \ref{t12}
and also crucial for the uniform global well-posedness in the next
section. We first show that $F^s(T)\hookrightarrow C([0, T], H^{s})$
for $s\in \R$, $T\in (0, 1]$ in the following proposition. We state
some results of Guo, Proposition \ref{p41}-\ref{p44}  can be found
in \cite{Guo}.
\begin{proposition}\label{p41}
If $s\in \R$, $T\in (0, 1]$ and  $u\in F^s(T)$, then
\begin{equation}
\sup_{t\in [0, T]}\norm{u(t)}_{H^s}\les \norm{u}_{F^s(T)}.
\end{equation}
\end{proposition}
\begin{proposition}\label{p42}
If $s\in \R$ and $u\in L_t^2H_x^s$, then
\begin{equation}
\norm{u}_{N^s}\les \norm{u}_{L_t^2H_x^s}.
\end{equation}
\end{proposition}
 We recall the
estimate in \cite{Guo} for the free solution.
\begin{proposition}\label{p43}
Let $s\in \R$. There exists $C>0$ such that for any $0\leq
\epsilon\leq 1$
\begin{equation}
\norm{\psi(t)W_\epsilon^\alpha(t)\phi}_{F^s}\leq C\norm{\phi}_{H^s},
\quad \forall \ \phi \in H^s(\R).
\end{equation}
\end{proposition}
Similarly for the inhomogeneous linear operator we have
\begin{proposition}\label{p44}
Let $s\in \R$. There exists $C>0$ such that for all $v \in
\Sch(\R^2)$ and  $0\leq \epsilon\leq 1$,
\begin{equation}
\norm{\psi(t)L(v)}_{F^s}\leq C \norm{v}_{N^s}.
\end{equation}
\end{proposition}

We next show \eqref{eq:MKdV-B} is uniformly (on $0<\epsilon \leq 1$)
locally well-posed in $H^s$, $s\geq1/4$. The procedure is quite
standard. See \cite{KPV}, for instance. By the scaling
\eqref{eq:scaling}, we see that $u$ solves \eqref{eq:MKdV-B} if and
only if $u_\lambda (x,t)= \lambda u(\lambda x, \lambda^3 t)$ solves
\begin{eqnarray}\label{eq:MKdV-B-m}
\partial_t u_\lambda + \partial^3_x u_{\lambda}+\epsilon \lambda^{4-2\alpha} |\partial_x|^{2\alpha}u_\lambda+6u^2_\lambda\partial_x (u_\lambda)=0, \ \
u_\lambda (0)= \lambda^2\phi (\lambda \, \cdot).
\end{eqnarray}
Since $s\geq1/4$,
\begin{equation}
\norm{\lambda^2\phi(\lambda
x)}_{H^s}=O(\lambda^{3/2+s}\norm{\phi}_{H^s}) \quad \mbox{as }
\lambda\rightarrow 0,
\end{equation}
thus we can first restrict ourselves to considering
\eqref{eq:MKdV-B} with data $\phi$ satisfying
\begin{equation}
\norm{\phi}_{H^s}=r\ll 1.
\end{equation}
We will mainly work on the integral equation
\begin{equation}
u(t)=W_\epsilon^\alpha(t)\phi_1-L \big(\partial_x (\psi^3
u^3)\big)(x,t) ,\label{eq:inteMKdV-buni}
\end{equation} and a truncated form
\begin{equation}
u(t)=\psi(t)\left[W_\epsilon^\alpha(t)\phi_1-L \big(\partial_x
(\psi^3 u^3)\big)(x,t) \right],\label{eq:truninteMKdV-buni}
\end{equation}
where $\psi$ is a smooth time cutoff function satisfying $\psi\in
C^{\infty}_{0}(\R)$, supp$\psi\subset[-2,2]$, $\psi\equiv1$ on
$[-1,1]$.

 We define the operator
\begin{equation}
\Phi_{\phi}(u)=\psi(t)W_\epsilon^\alpha(t)\phi- \psi(t)L
\big(\partial_x (\psi^3 u^3)\big),
\end{equation}
where $L$ is defined by \eqref{defL}. We will prove that $\Phi_\phi
(\cdot)$ is a contraction  mapping from
\begin{equation}
\mathcal{B}=\{w\in F^s:\ \norm{w}_{F^s}\leq 2cr\}
\end{equation}
into itself. From Propositions \ref{p42}, \ref{p43} and \ref{p44} we
get if $w\in \mathcal{B}$, then
\begin{eqnarray}
\norm{\Phi_\phi(w)}_{F^s}&\leq&
c\norm{\phi}_{H^s}+2\norm{\partial_x(\psi(t)^3w^3(\cdot, t))}_{N^s}\nonumber\\
&\leq& cr+2c\norm{w}_{F^s}^3\leq cr+2c(2cr)^3\leq 2cr,
\end{eqnarray}
provided $r$ satisfies $8c^3r^2\leq 1/2$. Similarly, for $w, h\in
\mathcal{B}$
\begin{eqnarray}
\norm{\Phi_\phi(w)-\Phi_\phi(h)}_{F^s}
&\leq& c\norm{ L \partial_x(\psi^3(\tau)(u^3(\tau)-h^3(\tau)))}_{F^s}\nonumber\\
&\leq&c(\norm{h^{2}(w-h)}_{F^s}+\norm{w(w^{2}-h^{2})}_{F^s})\nonumber\\
&\leq&c(\norm{u}_{F^s}\norm{w+h}_{F^s}\norm{w-h}_{F^s}+\norm{h}_{F^s}\norm{h}_{F^s}\norm{w-h}_{F^s})\nonumber\\
&\leq&8c^3r^2\norm{w-h}_{F^s}\leq \frac{1}{2}\norm{w-h}_{F^s}.
\end{eqnarray}
Thus $\Phi_\phi(\cdot)$ is a contraction. There exists a unique
$u\in \mathcal{B}$ such that
\begin{equation}
u=\psi(t)W_\epsilon^\alpha(t)\phi- 2\psi(t)L \big(\partial_x
(\psi^3 u^3)\big).
\end{equation}
Hence $u$ solves the integral equation \eqref{eq:inteMKdV-buni} in
the time interval $[0,1]$. Similar to Guo \cite{Guo}, we can show
that $u\in X^{1/2,s,\alpha}$. For general $\phi \in H^s$, by using
the scaling \eqref{eq:scaling} and the uniqueness result in
\cite{WL2}, we immediately obtain that Theorem \ref{t12} holds for
a small $T=T(\norm{\phi}_{H^s})>0$.

\section{Uniform global well-posedness for KdV-B equation}
In this section we will extend the uniform local solution obtained
in the last section to a uniform global solution. The standard way
is to use conservation law. We can verify that  if $v$ was a smooth
solution of \eqref{eq:MKdV}, then
\begin{equation}\label{eq:L2law}
H_{1}[v]=\int_\R (v_{x})^{2}-v^{4}+v^{2}dx
\end{equation}
is a conservation quantity for  \eqref{eq:MKdV}. However, there are
less symmetries for \eqref{eq:MKdV-B}. Let $u$ be a smooth solution
of \eqref{eq:MKdV-B}, we have
\begin{eqnarray}
\frac{d}{dt}H_{1}[u]&=&\int_\R
2u_{x}\partial_{x}u_{t}-4u^{3}u_{t}+2uu_{t}dx\nonumber\\
&=&\int_\R
2u_{x}\partial_{x}(-u_{xxx}-\varepsilon\abs{\partial_x}^{2\alpha}u+2(u^3)_x)-4u^{3}(-u_{xxx}-\varepsilon\abs{\partial_x}^{2\alpha}u+2(u^3)_x)\nonumber\\
&&+2u(-u_{xxx}-\varepsilon\abs{\partial_x}^{2\alpha}u+2(u^3)_x)dx\nonumber\\
&=&\int_\R
2u_{x}\partial_{x}(-\varepsilon\abs{\partial_x}^{2\alpha}u)+4u^{3}(\varepsilon\abs{\partial_x}^{2\alpha}u)-2u(\varepsilon\abs{\partial_x}^{2\alpha}u)dx\nonumber\\
&= &-2\varepsilon\int_\R
(\Lambda^{1+\alpha}u)^{2}dx+4\varepsilon\int_\R
u^{3}\Lambda^{2\alpha}udx-2\varepsilon\int_\R
\abs{\Lambda^{\alpha}u}^{2}dx\nonumber\\
&\leq&-\varepsilon\int_\R
(\Lambda^{2\alpha}u)^{2}dx+4\varepsilon\int_\R
u^{3}\Lambda^{2\alpha}udx
\end{eqnarray}
where we use the notation $\Lambda=|\partial_x|$. Using
Cauchy-Schwartz inequality, we have
\begin{eqnarray*}
4\varepsilon\int_\R
u^{3}\Lambda^{2\alpha}udx&\leq&4\varepsilon\norm{u^3}_2\norm{\Lambda^{2\alpha}u}_2\\
&\leq&4\varepsilon\norm{u}^3_6\norm{\Lambda^{2\alpha}u}_2\\
&\leq&4\varepsilon(4\norm{u}^6_6+\frac{1}{8}\norm{\Lambda^{2\alpha}u}^2_2)\\
&\leq&8\varepsilon\norm{u}^6_6+\frac{\varepsilon}{2}\norm{\Lambda^{2\alpha}u}^2_2
\end{eqnarray*}
Therefor, we have
\begin{equation*}
\frac{d}{dt}H_{1}[u]+\frac{\varepsilon}{2}\norm{\Lambda^{2\alpha}u}^2_2\leq\norm{u}^6_6
\end{equation*}
Using Galiardo-Nirenberg inequality
\[
\norm{u}_6^6 \les \norm{u}_2^{4}\norm{u_x}_{2}^{2}, \quad
\norm{u}_4^4 \les \norm{u}_2^{3}\norm{u_x}_2
\]
Hence, we get
\begin{equation}\label{here}
\sup_{[0,T]}\norm{u(t)}_{H^1}+\varepsilon^{\frac{1}{2}}(\int_0^T\norm{\Lambda^{2\alpha}u}_2^2d\tau)^{\frac{1}{2}}\leq
C(T,\norm{u_0}_{H^1})
\end{equation}
 By a standard limit
argument, \eqref{here} holds for $H^1$-strong solution. Thus if
$\phi\in H^1$, then we get that \eqref{eq:MKdV-B} is uniformly
globally well-posed.
\section{Limit Behavior}
In this section we prove Theorem \ref{t13}. From the remark
\ref{hr}, if we consider the limit behavior in $H^1$ we need a
$H^{2}$-conservation quantity. We first give a
$H^{2}$-conservation quantity for \eqref{eq:MKdV}. It is
well-known that the following KdV equation
\begin{equation}\label{eq:KdV}
u_t+u_{xxx}=3(u^2)_x, \ \ u(0)=\phi.
\end{equation}
 is completely integrable and has infinite
conservation laws. As a corollary one obtains that if $v$ was a
smooth solution to \eqref{eq:KdV}, then for any $k\in \Z_{+}$,
\begin{equation}
\sup_{t\in\R}\norm{v(t)}_{H^k}\les \norm{v_0}_{H^k}.
\end{equation}
Now we use Miura transform $\mathbf{M}$ to establish the relation of
\eqref{eq:KdV} and  \eqref{eq:MKdV}. If $v$ was a solution of
\eqref{eq:MKdV}, then $\mathbf{M}u=\partial_{x}v+v^2$ is a solution
of \eqref{eq:KdV}, see \cite{Tao3}. Using this fact, we can find a
$H^{2}$-conservation quantity of \eqref{eq:MKdV}. We can easily
verify that
\begin{equation}
H_{1}[u]=\int_{\R}(\partial_{x}u)^{2}+2u^{3}dx\label{hc}
 \end{equation} is a
$H^{1}$-conservation quantity of \eqref{eq:KdV}. Let
$u=\partial_{x}v+v^2$ in \eqref{hc},  we get
\begin{equation}
H_{2}[u]=\int_{\R}(u_{xx})^2+10u^2u_{x}^2+2u^6dx\label{hc2}
\end{equation}
is a $H^2$-conservation quantity of \eqref{eq:MKdV}. Obviously,
\eqref{eq:MKdV} has $L^{2}$-conservation law, so
\begin{equation}
H'_{2}[u]=\int_{\R}(u_{xx})^2+10u^2u_{x}^2+2u^6+u^2dx\label{hc3}
\end{equation}
is also a  $H^2$-conservation quantity of \eqref{eq:MKdV}.

 However, there are less
symmetries for \eqref{eq:MKdV-B}. We can still expect that the $H^k$
norm of the solution remains dominated for a finite time $T>0$,
since the dissipative term behaves well for $t>0$. We already see
that for $k=1$ from \eqref{eq:L2law}. Now we prove for $k=2$ which
will suffice for our purpose. We do not pursue for $k\geq 3$. Assume
$u$ is a smooth solution to \eqref{eq:MKdV-B}. By the equation
\eqref{eq:MKdV-B} and partial integration in \eqref{hc3}, we have
\begin{eqnarray*}
\frac{d}{dt}H'_{2}[u]&=&\int_\R
2u_{xx}\partial_{xx}(u_t)+20uu^{2}_{x}u_{t}+20u^{2}u_{x}\partial_{x}(u_{t})
+12u^{5}u_{t}+2uu_{t} dx\\
&=&\int_\R
2u_{xx}\partial_{xx}(-u_{xxx}-\varepsilon\abs{\partial_x}^{2\alpha}u+2(u^3)_x))dx\\
&&+\int_\R
20u^{2}u_{x}\partial_{x}(-u_{xxx}-\varepsilon\abs{\partial_x}^{2\alpha}u+2(u^3)_x)\\
&&+\int_\R(20uu^{2}_{x}+12u^{5}+2u)(-u_{xxx}-\varepsilon\abs{\partial_x}^{2\alpha}u+2(u^3)_x)dx\\
&=&-2\varepsilon\int_\R
(\abs{\partial_{x}}^{\alpha+2}u)^{2}dx-20\varepsilon\int_\R
uu^{2}_{x}(\abs{\partial_{x}}^{2\alpha}u)dx\\
&&-20\varepsilon\int_\R
u^{2}u_{x}\abs{\partial_{x}}^{2\alpha}u_{x}dx-12\varepsilon\int_\R
u^{5}(\abs{\partial_{x}}^{2\alpha}u)dx-2\varepsilon\norm{\abs{\partial_{x}}^{\alpha}u}^{2}_{2}\\
 &\leq&
 -2\varepsilon\norm{\Lambda^{\alpha+2}u}^{2}_{2}-2\varepsilon\norm{\Lambda^{\alpha}u}^{2}_{2}-
 20\varepsilon\int_\R
 uu^{2}_{x}\Lambda^{2\alpha}udx\\
 &&-20\varepsilon\int_\R
 u^{2}u_{x}\partial_{x}(\Lambda^{2\alpha}u)dx
-12\varepsilon\int_\R
 u^{5}\Lambda^{2\alpha}udx\\
&\lesssim&
 -\varepsilon\norm{\Lambda^{\alpha+2}u}^{2}_{2}-\varepsilon\norm{\Lambda^{\alpha}u}^{2}_{2}+
 \norm{uu^{2}_{x}}^{2}_{2}+\norm{u^{5}}^{2}_{2}+\norm{u^{2}u_{x}}^{2}_{2}+\norm{uu_{x}}^{2}_{2}+\norm{u}^{6}_{6}+\norm{u}^{2}_{2}
\end{eqnarray*}
where we use the notation $\Lambda=|\partial_x|$ and Cauchy-Schzrtz
inequality. Thus we have
\begin{eqnarray}
&&\frac{d}{dt}H'_{2}[u]+\frac{\epsilon}{2}
\norm{\Lambda^{2\alpha+1}u}_2^2 \nonumber\\
&\les&
\norm{uu^{2}_{x}}^{2}_{2}+\norm{u^{5}}^{2}_{2}+\norm{u^{2}u_{x}}^{2}_{2}+\norm{uu_{x}}^{2}_{2}+\norm{u}^{6}_{6}+\norm{u}^{2}_{2}\nonumber\\
&\lesssim&\norm{u}^{2}_{L^{\infty}}\norm{u_{x}}^{4}_{L^{4}}+\norm{u}^{4}_{L^{\infty}}\norm{u_x}^{2}_{2}+\norm{u}^{10}_{10}
+\norm{u}^{2}_{4}\norm{u_x}^{2}_{4}+\norm{u}^{2}_{4}\norm{u_{x}}^{2}_{2}\label{GN}
\end{eqnarray}
Using Galiardo-Nirenberg inequality, the right of \eqref{GN} can be
dominated by $\norm{u}_{H^2}$.

 Using Cauchy-Schwartz inequality, we
get
\begin{equation}\label{eq:H1law}
\sup_{[0,T]}\norm{u(t)}_{H^2}+\epsilon^{1/2}\brk{ \int_{0}^T
\norm{\Lambda^{2\alpha+1}u(\tau)}_2^2d\tau}^{1/2}\leq
C(T,\norm{\phi}_{H^2}), \quad \forall\ T>0.
\end{equation}
Assume $u_\epsilon$ is a $H^1$-strong solution to \eqref{eq:MKdV-B}
obtained in the last section and v is a $H^1$-strong solution to
\eqref{eq:MKdV} in \cite{Tao2}, with initial data $\phi_1,\phi_2\in
H^1$ respectively. We still denote by $u_\epsilon, v$ the extension
of $u_\epsilon, v$. From the scaling \eqref{eq:scaling}, we may
assume first that $\norm{\phi_1}_{H^1},\norm{\phi_2}_{H^1}\ll 1$.
Let $w=u_\epsilon-v$, $\phi=\phi_1-\phi_2$, then $w$ solves
\begin{eqnarray}\label{eq:diff}
\left \{
\begin{array}{l}
w_t+w_{xxx}+\epsilon |\partial_x|^{2\alpha}u_\epsilon=2(w(v^{2}+u^{2}_{\epsilon}+vu_{\epsilon}))_x, t\in \R_{+}, x\in \R,\\
v(0)=\phi.
\end{array}
\right.
\end{eqnarray}
We first view $\epsilon |\partial_x|^{2\alpha}u_\epsilon$ as a
perturbation to the difference equation of the MKdV equation.
Considering the integral equation of \eqref{eq:diff}
\begin{equation}
w(x,t)=W_0(t)\phi-\int_0^tW_0(t-\tau)[\epsilon
|\partial_x|^{2\alpha}u_\epsilon+2(w(v^{2}+u^{2}_{\epsilon}+vu_{\epsilon}))_x]d\tau,
\ t\geq 0.
\end{equation}
Then $w$ solves the following integral equation on $t\in [0,1]$,
\begin{eqnarray}
w(x,t)&=&\psi(t)[W_0(t)\phi-\chi_{\R_+}\int_0^tW_0(t-\tau)
(\tau)\psi(\tau)\epsilon
|\partial_x|^{2\alpha}u_\epsilon (\tau)d\tau\nonumber\\
&&\quad
-2\chi_{\R_+}\int_0^tW_0(t-\tau)(w(v^{2}+u^{2}_{\epsilon}+vu_{\epsilon}))_x(\tau)d\tau
].
\end{eqnarray}
By Proposition \ref{p42}, \ref{p43}, \ref{p44} and Lemma
 \ref{l41}, for $1/4\leq s\leq1$, we get
\begin{eqnarray}
\norm{w}_{F^s}&\les&
\norm{\phi}_{H^1}+\epsilon\norm{u_\epsilon}_{L^2_{[0,2]}\dot{H}_x^{2\alpha+s}}\nonumber\\
&&+\norm{w}_{F^s}\norm{u_{\varepsilon}}_{F^s}(\norm{v}_{F^s}+\norm{u_{\varepsilon}}_{F^s})+\norm{w}_{F^s}\norm{u_{\varepsilon}}_{F^s}\norm{v}_{F^s}.
\end{eqnarray}
Since from Theorem \ref{t12} we have
\[\norm{v}_{F^s}\les \norm{\phi_2}_{H^s}\ll 1,\quad \norm{u_\epsilon}_{F^s}\les \norm{\phi_1}_{H^s}\ll 1,\]
then we get that
\begin{equation}
\norm{w}_{F^s}\les
\norm{\phi}_{H^s}+\epsilon\norm{u_\epsilon}_{L^2_{[0,2]}\dot{H}_x^{2\alpha+s}}.
\end{equation}
From Proposition \ref{p41} and \eqref{eq:H1law} we get
\begin{equation}
\norm{u_\epsilon-v}_{C([0,1], H^s)}\les
\norm{\phi_1-\phi_2}_{H^s}+\epsilon^{1/2}C(\norm{\phi_1}_{H^2},\norm{\phi_2}_{H^1}).
\end{equation}
For general $\phi_1,\phi_2 \in H^1$, using the scaling
\eqref{eq:scaling}, then we immediately get that there exists
$T=T(\norm{\phi_1}_{H^1},\norm{\phi_2}_{H^1})>0$ such that
\begin{equation}\label{eq:limitL2}
\norm{u_\epsilon-v}_{C([0,T], H^s)}\les
\norm{\phi_1-\phi_2}_{H^s}+\epsilon^{1/2}C(T,\norm{\phi_1}_{H^2},\norm{\phi_2}_{H^1}).
\end{equation}
Therefore, \eqref{eq:limitL2} automatically holds for any $T>0$, due
to \eqref{eq:L2law} and \eqref{eq:H1law}.
\begin{proof}[Proof of Theorem \ref{t13}]
For fixed $T>0$, we need to prove that $\forall\ \eta>0$, there
exists $\sigma>0$ such that if $0<\epsilon<\sigma$ then
\begin{equation}\label{eq:limitHs}
\norm{S_T^\epsilon(\varphi)-S_T(\varphi)}_{C([0,T];H^s)}<\eta.
\end{equation}
We denote $\varphi_K=P_{\leq K}\varphi$, then we get
\begin{eqnarray}
&&\norm{S_T^\epsilon(\varphi)-S_T(\varphi)}_{C([0,T];H^s)}\nonumber\\
&\leq&\norm{S_T^\epsilon(\varphi)-S_T^\epsilon(\varphi_K)}_{C([0,T];H^s)}\nonumber\\
&&+\norm{S_T^\epsilon(\varphi_K)-S_T(\varphi_K)}_{C([0,T];H^s)}+\norm{S_T(\varphi_K)-S_T(\varphi)}_{C([0,T];H^s)}.
\end{eqnarray}
From Theorem \ref{t12} and \eqref{eq:limitL2}, we get
\begin{eqnarray}
\norm{S_T^\epsilon(\varphi)-S_T(\varphi)}_{C([0,T];H^s)}\les
\norm{\varphi_K-\varphi}_{H^s}+\epsilon^{1/2}C(T,K,
\norm{\varphi}_{H^s}).
\end{eqnarray}
We first fix $K$ large enough, then let $\epsilon$ go to zero,
therefore \eqref{eq:limitHs} holds.
\end{proof}
\noindent{\bf Acknowledgment.}  The author would like to express
his great thanks to Professor Baoxiang Wang and Doctor Zihua Guo
for their valuable suggestions and frequent encouragement during
every stage of this work. This work is supported in part by the
National Science Foundation of China, grant 10571004; and the 973
Project Foundation of China, grant 2006CB805902, and the
Innovation Group Foundation of NSFC, grant 10621061. 
\end{document}